\documentclass[reqno,a4paper,12pt]{amsart}
\usepackage{amssymb,amsmath,amscd,amstext,amsthm,amsfonts,mathtools,tikz}
\usepackage{graphicx}
\usepackage{setspace} 
\usepackage{subfig}
\usepackage[mathscr]{eucal}
\usepackage[ansinew]{inputenc} 
\usepackage{enumerate}

\usepackage{geometry}
\usepackage{amsrefs}
\usepackage{bbm}
\usepackage{float}
\usepackage{caption}
\usepackage{rotating}
\usepackage{fdsymbol}
\usepackage{mathdots}
\usetikzlibrary{matrix}
\usepackage{stmaryrd}
\usepackage{multirow}

\numberwithin{equation}{section}
\newtheorem{theorem}{Theorem}[section]

\newtheorem{corollary}[theorem]{Corollary}
\newtheorem{lemma}[theorem]{Lemma}

{\theoremstyle{definition}
{\newtheorem{remark}[theorem]{Remark}
\newtheorem{example}[theorem]{Example}

\newtheorem{defn}[theorem]{Definition}
}}

\newcommand{\Cc}{{\mathcal{C}}}
\newcommand{\Rr}{{\mathbb{R}}}
\def\diag{\operatorname{diag}}

\renewcommand{\d}{\mathrm d}

\newtheorem*{theorem*}{Theorem}

\newcommand{\B}{{\bf B}}
\newcommand{\D}{{\bf D}}
\newcommand{\W}{{\bf W}}

\newcommand{\Bt}{{\bf B}^t}
\newcommand{\Dt}{{\bf D}^t}

\newcommand{\0}{{\bf 0}}
\newcommand{\J}{\mathcal{J}}
\newcommand{\Dd}{\mathcal{D}}
\newcommand\scalemath[2]{\scalebox{#1}{\mbox{\ensuremath{\displaystyle #2}}}}
\newcounter{lst}
\newenvironment{lst}{%
\refstepcounter{lst}%
\begin{center}
\begin{minipage}{.9\textwidth}}{%
\end{minipage}%
\makebox[.1\textwidth][r]{(\thelst)}%
\end{center}}
\begin{document}
\title[The Hamiltonian formalism of the inverse problem]{The Hamiltonian formalism of the inverse problem}
\date{\today}
\keywords{Hamiltonian Inverse Problem, Dirac Structure, Big-isotropic Structure, Linear Systems, Completely Integrable}


\author[H. N. Alishah]{Hassan Najafi Alishah}
\address{Departamento de Matem\'atica, Instituto de Ci\^encias Exatas\\
Universidade Federal de Minas Gerais \\
Belo Horizonte, 31270-901, Brazil}
\email{halishah@mat.ufmg.br}

\begin{abstract}
A big-isotropic structure is a generalization of the notion of Dirac structure, due to Vaisman. We discuss the inverse problem of deciding if a vector field is Hamiltonian having a big-isotropic structure as underlying geometry. In \cite{hassan-2020-2} we have considered this question for the special case of replicator equations. Here we generalize that approach to any vector field that can be written in the form $X(u)=(\B\eta)(u),$ where $\B=\B(u)$ is a matrix and  $\eta=\eta(u)$ is a vector. For a linear system we show that, if the representing matrix of the system has at least one pair of positive-negative non-zero eigenvalues, or a zero eigenvalue with at least one 3-dimensional Jordan block associated to it, then the linear system has a Hamiltonian description with respect to a big-isotropic structure. As a byproduct, we find a class of linear systems with zero eigenvalue that are Hamiltonian with respect to a big-isotropic structure but not a Dirac structure. Moreover, we prove that every linear Hamiltonian system, having a big-isotropic structure as underlying geometry, is completely integrable in the sense of Zung \cite{zung-action-angle}. For linear systems, in the both Hamiltonian formulation and complete integrability cases, explicit descriptions of the geometric structures, the Hamiltonian functions, the first integrals and the commuting flows are provided. 
\end{abstract}

\maketitle

\section{Introduction}\label{introduction}
Given equations of motion, say $\ddot{u}=F(u,\dot{u})$, the inverse problem in the Lagrangian formalism aims to find a Lagrangian function $\mathcal{L}=\mathcal{L}(u,\dot{u})$ such that the given equations of motion are the Euler-Lagrange equations of $\mathcal{L}$. In the Hamiltonian version of this problem, in addition to finding a Hamiltonian function, one has to determine the underlying geometric structure as well. If the Lagrangian function is regular it gives rise, through the Legendre transformation, to a Hamiltonian description of the given equations of motion on the cotangent bundle, where the underlying geometric structure is the the canonical symplectic structure of the cotangent bundle. However, for many interesting problems the Lagrangian is not regular and the Hamiltonian description cannot be obtained from the Lagrangian description. For this reason the Hamiltonian inverse problem should be considered independently -- see \cite[Introduction]{MR1233228} for more on this matter.  

The Hamiltonian formalism of the inverse problem has been studied in the context of symplectic and Poisson geometries, see e.g. \cite{MR1233228,geo-from-dyn}. Here, we consider as underlying geometry for our Hamiltonian descriptions Dirac structures and their generalizations, the so-called big-isotropic structures (see Section~\ref{dirac-bigisotropic} for background). It includes as special cases symplectic, presymplectic and Poisson structures. We only consider here the local problem, so we work on $\mathbb{R}^m$. The global aspects of the problem will be addressed in future work. 

In a recent paper \cite{hassan-2020-2} we have studied this question for replicator equations (or, equivalently, Lotka-Volterra equations) where the pay-off matrix is not skew-symmetrizable. By using big-isotropic structures we were able to enlarge the set of conservative replicator and Lotka-Volterra equations that admit a Hamiltonian formulation. Here, we extend this approach.

In order to explain our main idea, let $X=\B\eta$ be a vector field in $\mathbb{R}^m$ that can be factored as a product of a matrix $\B=\B(u)$ and a vector $\eta=\eta(u)$. We will be looking for a matrix valued function $\D=\D(u)$ such that
\begin{enumerate}[(1)]
\item the $1$-form $\Dt\eta$ is closed,
\item the matrix $\D\B(u)$ is skew-symmetric for every $u\in\mathbb{R}^m$, 
\item an integrability condition is satisfied, (condition (ii) of Lemma~\ref{lemma:our-dirac-structure}).
\end{enumerate}
 Items $(1)$ and $(2)$ lead to first integrals of the motion for $X$ which are the candidates for the Hamiltonian we need. Furthermore, in Lemma~\ref{lemma:our-dirac-structure} we show that when items $(2)$ and $(3)$ are satisfied, the pair $(\B,\D^t)$ generates a  big-isotropic structure, which is Dirac if and only if $\ker\B\cap\ker\D^t=0$. Corollary~\ref{non-singular-cases} provides the criteria to further determine if the structure generated by $(\B,\D^t)$ is symplectic, presymplectic or Poisson.  Our Theorem~\ref{main-result} states that if all the three items are satisfied, then the vector field $X$ is Hamiltonian with the Hamiltonian function $H$ defined by $dH=\Dt\eta$ and the underlying big-isotropic structure generated by the pair $(\B,\D^t)$. In the case of replicator equations and linear systems, the problem reduces to finding a particular type of constant matrices $D$ such that $DB$ is skew-symmetric. We believe that our approach can be applied to other problems as well, beyond the replicator equations treated in \cite{hassan-2020-2}.

Linear Hamiltonian systems on $\mathbb{R}^{2k}$ relative to the canonical symplectic form were first studied and classified by Williamson \cite{Williamson}. One of the main steps in his approach was to transform the linear system $X=Bu$ into simpler form, diagonal or  Jordan normal form, preserving the symplectic structure, i.e., using canonical changes of variables. The factorization approach presented in \cite{MR1233228} and discussed in more detail in the textbook~\cite{geo-from-dyn}, overcomes this problem. In our approach the matrix $B$ can simply be considered to be in Jordan canonical form or any other member of its conjugacy class -- see Remark~\ref{normalizing-to-jordan}. 

The factorization of \cite{MR1233228,geo-from-dyn} amounts to finding a skew-symmetric matrix $\Lambda$ and a symmetric matrix $H$ such that $B=\Lambda H$. The matrix $\Lambda$ represents a Poisson structure while $H$ yields a quadratic Hamiltonian.  When $B$ is non-singular $\Lambda$ becomes a symplectic structure, and for this it is necessary to have non-zero positive-negative eigenvalues in pairs with the same Jordan blocks, see~\cite[Theorem 4.2]{geo-from-dyn}. This shows that the scope of applicability of this approach is limited. For example, it does not detect if the linear system is Hamiltonian w.r.t.\footnote{with respect to} a presymplectic structure, where zero eigenvalues can be present and the size of Jordan blocks can be different. Our approach removes these restrictions and is able to detect possible symplectic, presymplectic, Poisson and Dirac structures. Even more general, it is able to detect possible big-isotropic structures. We we will see, for example, that a linear system that contains a single even dimensional Jordan block associated to a zero eigenvalue has a Hamiltonian description w.r.t. a big-isotropic structure but not a Dirac structure -- see Corollary~\ref{cor:only-proper-big-iso}. Furthermore, Casimirs and isotropic vector fields can be detected -- see Examples~\ref{example:zero-eigenvalue} and~\ref{example:nonzero-eigenvalue}. A given linear system can have more than one Hamiltonian description and another advantage of our approach is that the reason for this lack of uniqueness becomes very clear -- see Remark~\ref{remark-main-results-linear}. 

Finally, Zung in \cite{zung-action-angle} has introduced a notion of complete integrability for arbitrary vector fields, Hamiltonian or not. We will show that a linear Hamiltonian system, in the big-isotropic sense, is always completely integrable in the sense of Zung. This extends a similar result obtained in \cite{MR1233228} for a more restricted class of systems. 

{\bf Organization of the paper.} 
In Section~\ref{dirac-bigisotropic}, we provide a brief introduction to various geometric structures, such as symplectic, Poisson and Dirac structures, and how they fit in the general set up of big-isotropic structures on $\mathbb{R}^m$. In Section~\ref{hamiltonian-inverse-problem}, we state and prove our results on the Hamiltonian formalism of the inverse problem for general systems. In Section~\ref{sec:linear-vector-fields}, we discuss the case of linear Hamiltonian systems. In Section~\ref{sec:integrability}, we start by providing a brief introduction to integrable systems and then we show that every linear Hamiltonian system is completely integrable in the sense of Zung.


\section{Dirac and big-isotropic structures} 
\label{dirac-bigisotropic}
In this section, we provide a brief introduction to big isotropic structures. We start with the geometric structures which are special cases of big-isotropic structure. 

{\bf Symplectic structure:} A non-degenerate closed $2$-form $\omega$ on a manifold $M$ is called a \emph{symplectic structure.} Non-degenerate refers to the fact that the vector bundle map $\omega^\sharp:TM\to T^\ast M$ defined by $X\mapsto \omega(X,.)$ is non-degenerate. This condition forces $M$ to be even dimensional. Relaxing the non-degeneracy condition on $\omega$, the closed $2$-form defines a \emph{presymplectic structure} on $M$. A presymplectic manifold doesn't need to be even dimensional. The map $\omega^\sharp$ is used to define the Hamiltonian vector filed $X_H$ of a given function $H$ by $\omega^\sharp(X_H)=d H$. We will use the same notation $\omega^\sharp$ for the representing matrix, in local coordinates, of the vector bundle map $\omega^\sharp$. 

{\bf Poisson structure:} Let $\pi$ be a bivector on $M$ i.e., a bilinear, antisymmetric map $\pi:T^\ast M\times T^\ast M\to\mathbb{R}$. If $\pi$ satisfies $[\pi,\pi]=0$ where $[.,.]$ is the Schouten bracket, then it defines a \emph{Poisson  structure} on $M$. An alternative definition is via a Poisson  bracket i.e., a bilinear skew-symmetric bracket  $\{.,.\}$ on the space of smooth function $C^\infty(M)$ which satisfies the Leibniz's rule and the Jacobi identity. These two definitions are connected by $\{f,g\}=\pi(df, dg),\,\forall f,g,\in C^\infty(M)$.  Similar to the $2$-form $\omega$, the bivector $\pi$ defines 
a vector bundle map $\pi^\sharp:T^\ast M\to T M$ by $\alpha\mapsto\pi(\alpha,.)$. This map is used to associate to a given function $H$ its Hamiltonian vector field by $ X_H=\pi^\sharp(d H)$. The Hamiltonian evolutionary games discussed in \cite{AD2014} are of this type. 
The local expression for the Jacobi identity (equivalently, closeness w.r.t  the Schouten bracket) is
\begin{equation}
\label{jacobi-condition-Poisson} \sum_{l=1}^m (\pi^\sharp_{ij}\frac{\partial \pi^\sharp_{ik}}{\partial u_l}+\pi^\sharp_{li}\frac{\partial \pi^\sharp_{kj}}{\partial u_l}+\pi^\sharp_{lk}\frac{\partial \pi^\sharp_{jl}}{\partial u_l})=0\quad\forall i,j,k.
\end{equation}
 The Jacobi identity~\eqref{jacobi-condition-Poisson} guaranties integrability of the possibly non-singular distribution defined the image of $\pi^\sharp$. This foliation is referred  to as \emph{characteristic foliation of $\pi$}. Each leaf of this foliation have a symplectic structure induced by $\pi$. The rank of $\pi$ at a given point $m$ is defined to be the rank of its characteristic foliation at $m$. The characteristic foliation is invariant under the flow of any Hamiltonian vector field $X_H$.

{\bf Dirac structure:} Dirac structure, introduced in \cites{MR998124,MR951168}, unites and generalizes Poisson  and presymplectic structures.
 For a manifold $M$, the vector bundle $\mathbb{T}M= TM\oplus T^\ast M$ is called \emph{the generalized tangent bundle of $M$}. In order to simplify notations, we will use the same notation for the generalized tangent bundle and the space of its sections i.e.,  the set
 $$\{(X,\alpha)\,,\mbox{where $X$ is a vector field on $M$ and $\alpha$ a $1$-form}\}.$$
  The same will hold for any subbundle $L$ of $\mathbb{T}M$.  The natural pairing
\begin{equation}
\label{pairing}
\langle (X,\alpha),(Y,\beta)\rangle:=\frac{1}{2}\left(\beta(X)+\alpha(Y)\right)\, \forall (X,\alpha),(Y,\beta)\in\mathbb{T}M,
\end{equation} 
defines a bilinear form on the space of the sections of $\mathbb{T}M$. For a given linear subbundle  $L$ of $\mathbb{T}M$, its annihilator w.r.t. the bilinear form $\langle.,.\rangle$ is defined as 
\[L^\perp:=\{(X,\alpha)\in \mathbb{T}M\,\,| \,\,\langle(X,\alpha),(Y,\beta)\rangle=0\quad \forall (Y,\beta)\in L\}.\]
The form $\langle.,.\rangle$ is neither positive definite nor negative definite. As a consequence the intersection $L\cap L^\perp$ can be non-empty.  Having this in mind, a linear subbundle $L$ of $\mathbb{T}M$ is called \emph{isotropic} if $L\subseteq L^\perp$ and \emph{maximal isotropic} when $L=L^\perp$. Maximal isotropy implies that the dimension of the fibers of $L$ is equal to the dimension of $M$.

\begin{defn}
A \emph{Dirac structure} on a manifold $M$ is a maximal isotropic linear subbunlde $L$ of $TM\oplus T^\ast M$  such that for any two sections  $(X,\alpha),(Y,\beta)\in L$ 
 \begin{equation}\label{courant-bracket}
\llbracket(X,\alpha),(Y,\beta)\rrbracket:=\left([X,Y], \mathcal{L}_X\beta-\mathcal{L}_Y\alpha+\frac{1}{2}d(\alpha(Y)-\beta(X))\right)\in L,
\end{equation}
where $[.,.]$ denotes the Lie bracket between vector fields and $\mathcal{L}$ stands for the Lie derivative. 
The bracket $\llbracket .,.\rrbracket$ is called \emph{the Courant bracket}.

 Furthermore, a vector field $X$  is called \emph{Hamiltonian} w.r.t. the Dirac structure $L$ if there exists a Hamiltonian  $H\in C^\infty(M)$ such that ${(X,d H)\in L.}$ 
 \end{defn}
 
 The first examples of Dirac structure to mention are
 \begin{itemize}
 \item[$\bullet$] The graph of a presymplectic form $\omega$ defined by 
 \[L_\omega:=\{(X,\omega^\sharp(X)),\, \mbox{for any vector field $X$}\}. \]
 Skew-symmetricness of $\omega$ yields  the maximal isotropy condition and  $\omega$ being close yields~\eqref{courant-bracket}.
 \item[$\bullet$] The graph of a Poisson bivector $\pi$ defined by 
 \[L_\pi:=\{(\pi^\sharp(\alpha),\alpha),\mbox{for any $1$-form $\alpha$ }\}.\]
  Similarly, skew-symmetricness of $\pi$ and the Jacobi identity ~\eqref{jacobi-condition-Poisson}  yield the requirements for being a Dirac structure.  
 \end{itemize}
 These  examples show that Dirac structure unifies presymplectic and Poisson  structures. 
 
 The following lemma shows that closeness w.r.t. the Courant bracket, in spite of its nonlinear character,  can be verified on a basis of a given isotropic subbundle $L$. 

 \begin{lemma}\label{check-on-basis}
Let $f\in C^\infty(M)$ and $(X,\alpha),(Y,\beta)\in \mathbb{T}M$. Then 
\begin{equation}\label{check-on-basis-1}
\llbracket(X,\alpha),f(Y,\beta)\rrbracket=f\llbracket(X,\alpha),(Y,\beta)\rrbracket+X(f).(Y,\beta)-\langle (X,\alpha),(Y,\beta)\rangle (0,df).
\end{equation}
\end{lemma}
\begin{proof}
\begin{align*}
\llbracket(X,\alpha),f(Y,\beta)\rrbracket&=([X,fY],\mathcal{L}_X(f\beta)-\mathcal{L}_{fY}\alpha+\frac{1}{2}d\left(\alpha(fY)-f\beta(X)\right)\\
&=(f[X,Y]+X(f)Y,X(f)\beta+f\mathcal{L}_X \beta -f\mathcal{L}_{Y}\alpha-\alpha(Y).df\\
&+\frac{1}{2}f.d\left(\alpha(Y)-\beta(X)\right)+\frac{1}{2}\left(\alpha(Y)-\beta(X)\right).df)\\
&=f\llbracket(X,\alpha),(Y,\beta)\rrbracket+X(f).(Y,\beta)-\langle (X,\alpha),(Y,\beta)\rangle (0,df).
\end{align*}
\end{proof}

\begin{remark}\label{check-on-basis-3}
Note that if $L$ is isotropic, then the last term on the right hand side of equation~\eqref{check-on-basis-1} will vanish. Consequently, if for  $(X,\alpha),(Y,\beta)\in L$ we have that $\llbracket(X,\alpha),(Y,\beta)\rrbracket\in L$, then  $\llbracket(X,\alpha),f(Y,\beta)\rrbracket\in L$. 
\end{remark}

We now define the notion of  big-isotropic structure which generalizes all the previous structures. This notion is firstly defined and studied by Vaisman  in \cite{MR2343378,MR2349409}. Up to our knowledge, these two papers and author's preceding work~\cite{hassan-2020-2} are the only places where big-isotropic structures have appeared. 

\begin{defn}  \emph{A big-isotropic structure} is an isotropic linear subbundle $L$ of $TM\oplus T^\ast M$ such that for any two sections  $(X,\alpha),(Y,\beta)\in L$ their Courant bracket \eqref{courant-bracket} is also a section of $L$.
\end{defn}

As in the case of Dirac structure, a vector field $X\in\mathfrak{X}(M)$  is called \emph{Hamiltonian} w.r.t. the big-isotropic structure $L$ if there exists a Hamiltonian  $H\in C^\infty(M)$ such that $(X,d H)\in L.$ 
\begin{remark}\label{weak-hamiltonian}
In \cite{MR2349409} Vaisman defines the notion of a weak-Hamiltonian vector field in the context of big-isotropic structures  where a vector field $X$ is called weak-Hamiltonian with a Hamiltonian function $H$ if $(X,d H)\in L^\perp.$ There, he shows that  port-hamiltonian systems from control theory could be formalized as weak-Hamiltonian systems in this setting. Here, we consider Hamiltonian vector fields and not the weak ones but this is a direction that can be explored more in a future work.
\end{remark}

\begin{defn}\label{casimirs}  A function $F$ is called a Casimir of the big-isotropic structure $L$ if $(0,d F)\in L$ and a vector field $X$ is called isotropic w.r.t. $L$ if $(X,0)\in L$. Note that Casimirs are constants of motion for every Hamiltonian vector fields. Furthermore, adding an isotropic vector field $X$ to a Hamiltonian vector field $X_H$ yields an other Hamiltonian vector field with the same Hamiltonian function. 
\end{defn}
   
 Similar to Poisson structure, the possibly  nonsingular distribution $P_1(L)$, where $P_1$ is the projection from $\mathbb{T}M$ to the tangent bundle $TM$, of a big-isotropic structure $L$  is integrable  since the sections of $L$ are closed under the Courant bracket.  For this reason, closeness w.r.t. the Courant bracket is also referred to as integrability condition.  Every leaf $S$ of the foliation tangent to  $P_1(L)$ is equipped with the closed $2$-form
\begin{equation}\label{presymplectic-on-leaves}
\omega_S(P_1(X,\alpha),P_1(Y,\beta))=\frac{1}{2}\left(\alpha(Y)-\beta(X)\right)\quad \forall P_1(X,\alpha),P_1(Y,\beta)\in TS.
\end{equation}
The  presymplectic foliation defined this way is called \emph{characteristic foliation} of the big-isotropic structure $L$.

\section{Hamiltonian inverse problem}\label{hamiltonian-inverse-problem}
In this section we use big-isotropic structures to address the Hamiltonian inverse problem for a given vector field  $$X(u)=(\B\eta)(u),$$ where $\B=\B(u)$ is a matrix and  $\eta=\eta(u)$ is a vector. All over this paper, we will think  of $\eta$ as a $1$-form on $\mathbb{R}^m$. We start by introducing the type of big-isotropic structures we will be using. 
\begin{lemma}\label{lemma:our-dirac-structure}
Let $\B=\B(u)$ and $\D=\D(u)$ be two $(m\times m)$-matrices. We define the subbundle $L_{(\B,\Dt)}$ of  $T\mathbb{R}^m\oplus T^\ast \mathbb{R}^m$ by $$L_{(\B, \Dt)}(u):=\{({\bf B}(u)z,{\bf D}^t(u)z)|\quad \forall z\in\mathbb{R}^m\}\quad \forall u\in\mathbb{R}^m.$$
 The subbundle $L_{(\B,\Dt)}$ is a big-isotropic structure if and only if
\begin{enumerate}[(i)]
\item $(\D \B+\Bt\Dt)(u)=0$ for every point $u\in\mathbb{R}^m$.
\item For every $i,j=1,\ldots,m$ we have 
\[ \llbracket(E_i,\zeta_i),(E_j,\zeta_j)\rrbracket \in L_{(\B,\Dt)}\quad\mbox{where}\quad (E_i,\zeta_i)=(\B e_i,\Dt e_i)\,\,\forall i\]
\end{enumerate}
Furthermore it is a Dirac structure if it also satisfies 
\begin{enumerate}
\item[(iii)] $(\ker \B(u))\cap(\ker  \Dt(u))=0$ for every point $u\in\mathbb{R}^m.$ 
\end{enumerate}
\end{lemma}
\begin{proof}
Item $(i)$ is equivalent to $L_{(\B,\Dt)}\subset L^\perp_{(\B,\Dt)}$ and item $(iii)$ yields $L_{(\B,\Dt)}=L^\perp_{(\B,\Dt)}$. Item $(ii)$ says that the integrability condition~\eqref{courant-bracket} is satisfied on the basis ${(E_i,\zeta_i)\, i=1,\ldots, n}$ and we have already showed in Lemma~\ref{check-on-basis} that this is enough to guarantee the integrability condition for $L_{(\B,\Dt)}$.
\end{proof}

\begin{remark}
Locally, any Dirac structure is of the type  introduced in Lemma~\ref{lemma:our-dirac-structure}-- see \cite[Section $1.2$]{MR998124} . 
\end{remark}
 
 The following corollary will be used to put a given big-isotropic structure $L_{(\B,\Dt)}$ in a normal form when the big-isotropic structure is a graph of a symplectic, presymplectic, or Poisson structures. 

\begin{corollary}\label{non-singular-cases} 
Let $L_{(\B,\Dt)}$ be a big-isotropic structure, as defined in Lemma~\ref{lemma:our-dirac-structure}, and $\W=\W(u)$  a non-singular $(m\times m)$-matrix. Then  $L_{(\B{\bf W},\Dt{\bf W})},$ is another representation of $L_{(\B,\Dt)}$. Consequently, 
\begin{enumerate}[(i)]
\item If $\B$ is non-singular, then $L_{(\B,\Dt)}=L_{(I, \Dt(\B)^{-1})}$ is the graph of the presymplectic form $\omega^\sharp=\Dt(\B)^{-1}$.  
\item If $\Dt$ is non-singular, then $L_{(\B,\Dt)}=L_{(\B(\Dt)^{-1},I)}$ is the graph of the  Poisson structure $\pi^\sharp= \B(\Dt)^{-1}$. 
\end{enumerate}
\end{corollary}
\begin{proof}
It is clear that as a subbundle $L_{(\B,\Dt)}=L_{(\B{\bf W},\Dt{\bf W})}$. This change of representation doesn't violate the requirements for the subbundle to be big-isotropic. The subbundle $L_{(\B{\bf W},\Dt{\bf W})}$ being isotropic is due to the fact
\[\D\B\quad\mbox{ is skew-symmetric}\quad\Leftrightarrow  \quad \W^t\D\B\W\quad \mbox{ is skew-symmetric},\]
and the integrability condition is satisfied thanks to the linear property mentioned in Remark~\ref{check-on-basis-3}. Also, note that this change of representation neither changes the characteristic foliation, nor the presymplectic structures on its leafs. 
\end{proof}

The change of representation, discussed in Corollary~\ref{non-singular-cases}, is introduced in~\cite[Theorem $1.3.3$]{MR998124} in a different form. It also used in~\cite{MR2422350} for finding a local normal form of Dirac structures.

Given a big-isotropic structure $L_{({\bf B},{\bf D}^t)}$, a pair $(X,d H)$ is Hamiltonian w.r.t. it  if and only if there exists a function $\eta:\mathbb{R}^m\to \mathbb{R}^m$ such that $X={\bf B}\eta$ and $dH={\bf D}^t\eta$. Now, we are ready to state the main result of this section. 

\begin{theorem}\label{main-result}
Let $X=\B\eta$ be a vector field on $\Rr^m$. If there exists a matrix $\D=\D(u)$  such that  
\begin{enumerate}[(1)]
\item the $1$-form $\Dt\eta$ is closed,
\item $\D\B(u)$ is skew-symmetric for any $u\in\Rr^m$,          
\end{enumerate}
then the function $H$ defined by $dH=\Dt\eta$ is a constant of motion for the vector field $X$. In addition, if  for any $i,j=1,\ldots,m$ there exists $c_{ij}\in\mathbb{R}^m$ such that 
\[\llbracket (E_i,\zeta_i),(E_j,\zeta_j)\rrbracket=(\B c_{ij},\Dt c_{ij}),\]
where $E_i=\B e_i$ and $\zeta_i=\Dt e_i$, then $X$ is Hamiltonian having the function $H$ as Hamiltonian function and the big-isotropic structure $L_{(\B,\Dt)}$ as underlying structure. Furthermore, 
\begin{enumerate}[(i)]
\item if $\ker\B(u)\cap \ker\D^t(u)=0, \forall u\in\mathbb{R}^m$, then $L_{(\B,\Dt)}$ is a Dirac structure.  
\item if the matrix $\B(u)$ is non-singular, then  the underlying structure is the presymplectic structure $\omega^\sharp=\Dt(\B)^{-1}$,
\item if the matrix $\Dt(u)$ is non-singular, then the underlying structure is the Poisson structure $\pi^\sharp= \B(\Dt)^{-1}$,
\item the intersection of items (ii) and (iii) would yield a symplectic structure as underlying geometry. 
\end{enumerate}
\end{theorem}
\begin{proof}
The function $H$ is a constant of motion since
\[<X, dH>=<\B\eta,\D^t\eta>=\eta^t\D\B\eta=0,\]
where we used the fact that $\D\B$ is skew-symmetric. The rest of the statements are immediate from the definition of Hamiltonian systems in the context of big-isotropic structures, Lemma~\ref{lemma:our-dirac-structure} and Corollary~\ref{non-singular-cases}. 
\end{proof}

\section{Linear Hamiltonian Systems}\label{sec:linear-vector-fields}
 In this section, we apply our method to a given linear vector field
\begin{equation}\label{constant-gradiant}
X(u)=Bu,\quad u\in\mathbb{R}^m,
\end{equation}
 where $B$ is a constant $m\times m$-matrix.  Since $B$ is constant, if we find a constant $D$  such that the matrix $DB$ is skew-symmetric, then the  pair $(B,D^t)$ generates the big-isotropic structure
 $$L_{(B,D^t)}(u):=\{(Bz,D^tz) | \forall z\in\mathbb{R}^m\}.$$
 In other words, the integrability condition is satisfied automatically. 
 Theorem~\ref{main-result} also requires that the $1$-form $D^tu$ be closed. This requirement forces $D$ to be symmetric. The following corollary of Theorem~\ref{main-result} sums up what we have said. 
\begin{corollary} \label{cor:linear-vector-fileds}
Let $B$ be a constant matrix. If there exists a constant matrix $D$ such that $DB$ is skew-symmetric, then $L_{(B,D^t)}(u):=\{(Bz,D^tz) | \forall z\in\mathbb{R}^m\}$ is a big-isotropic structure. Furthermore, if $D$ is symmetric, then the linear system $X(u)=Bu$ is Hamiltonian w.r.t the big-isotropic structure $L_{(B,D^t)}(u)$ and  the Hamiltonian function $H(u)=u^t D u$. 
\end{corollary}

Clearly, items (i)-(iv) of  Theorem~\ref{main-result} hold as well. Also, note that every element $\eta\in\ker B$ which does not belong to $\ker D$ yields a non-trivial linear Casimir $H_\eta$ defined by $d H_\eta= D\eta$ and every element $\xi\in \ker D$ which is not in  $\ker B$ yields a non-trivial isotropic vector filed $X_\xi=B\xi$. 

\begin{remark}\label{normalizing-to-jordan}
Under a linear change of coordinates $u=Tv$ we have
\[X(v)=T^{-1}B T v,\quad\mbox{and}\quad H(v)=v^tT^tD T v,\]
i.e., the matrices $B$ and $D$ are, respectively, transformed to $T^{-1}B T$ and $T^tD T$. Clearly, 
\begin{itemize}
\item[-] $DB$ is skew-symmetric if and only if $T^tDT T^{-1}BT$ is so.
\item[-] $D$ is symmetric if and only if $T^t DT$ is so. 
\end{itemize}
Then, instead of $B$ we could consider any other member of its conjugacy class. Here, we will use the Jordan normal form of $B$. 
\end{remark}

In order to sate our next result we need to set some notations. The notations $I_n$ and $\0_{m\times n}$ will, respectively, denote the $n$ dimensional identity matrix and the $m\times n$-dimensional zero matrix. The dimension subscripts will be removed when there is no ambiguity. We will also use the notation ${\rm diag}(A_1,A_2,...,A_k)$ for a matrix with diagonal blocks $A_1,A_2,...,A_k$.  The notation $\J(\lambda)$ will stand for the Jordan block of a real eigenvalue $\lambda$. In general the Jordan block $\J(\lambda)$ has the form $$\J(\lambda)=\diag(\lambda I_r,J_{s_1}(\lambda),\ldots,J_{s_k}(\lambda))$$ where 
\begin{equation}\label{jordan-block-real}
J_{s_j}(\lambda) =
\begin{tikzpicture}[baseline=(current bounding box.center)]
\matrix (m) [matrix of math nodes,nodes in empty cells,right delimiter={)},left delimiter={(} ]{
\lambda&1&0& &0\\
0&\lambda&1& &0\\
& & &&\\ 
0&0&&\lambda&1\\
0&0&&0&\lambda\\
} ;
\draw[loosely dotted] (m-1-3)-- (m-1-5);
\draw[loosely dotted] (m-2-3)-- (m-2-5);
\draw[loosely dotted] (m-2-1)-- (m-4-1);
\draw[loosely dotted] (m-2-2)-- (m-4-2);
\draw[loosely dotted] (m-2-5)-- (m-4-5);
\draw[loosely dotted] (m-2-2)-- (m-4-4);
\draw[loosely dotted] (m-2-3)-- (m-4-5);
\draw[loosely dotted] (m-4-2)-- (m-4-4);
\draw[loosely dotted] (m-5-2)-- (m-5-4);
\end{tikzpicture}_{s_j\times s_j},
\end{equation}
Similarly, $\J(a\pm bi)$ will be used for the Jordan block of a complex eigenvalue $(a\pm b i)$. The general form of this block is 
$$\J(a\pm bi)=\diag(\oplus^rB_{(a\pm bi)},J_{2s_1}(a\pm bi),\ldots,J_{2s_k}(a\pm bi)),$$
where $\oplus^rB_{(a\pm bi)}=\diag(B_{(a\pm bi)},\ldots,B_{(a\pm bi)})$,
$ B_{(a\pm bi)}=\begin{pmatrix}a&b\\-b&a\end{pmatrix}$ and
\begin{equation}\label{jordan-block-complex}
J_{2s_j}(a\pm bi)= \begin{tikzpicture}[baseline=(current bounding box.center)]
\matrix (m) [matrix of math nodes,nodes in empty cells,right delimiter={)},left delimiter={(} ]{
B_{(a\pm bi)}&I_2&\0&&\0\\
\0&B_{(a\pm bi)}&I_2&&\0\\
&&&&\\
\0&\0&&B_{(a\pm bi)}&I_2\\
\0&\0&&\0&B_{(a\pm bi)}\\} ;
\draw[loosely dotted] (m-1-3)-- (m-1-5);
\draw[loosely dotted] (m-2-3)-- (m-2-5);
\draw[loosely dotted] (m-2-1)-- (m-4-1);
\draw[loosely dotted] (m-2-2)-- (m-4-2);
\draw[loosely dotted] (m-2-5)-- (m-4-5);
\draw[loosely dotted] (m-2-2)-- (m-4-4);
\draw[loosely dotted] (m-2-3)-- (m-4-5);
\draw[loosely dotted] (m-4-2)-- (m-4-4);
\draw[loosely dotted] (m-5-2)-- (m-5-4);
\end{tikzpicture}_{2s_j\times 2s_j}.
\end{equation}
 
 \begin{remark}\label{noation-pm}
With an abuse of notation, we will use the same letters for the possibly different dimensions of the Jordan blocks. This would normally happen for the eigenvalues with different absolute values. We  will also abuse the notation $d$ to make the equations more readable. 
\end{remark}

We need to set one more notation.  For a given number $m$, we will use the notation $K_m(l)$, $l=1\ldots m$, for the matrix with all its elements equal to zero except the $l^{\rm th}$ inferior anti-diagonal which contains $1$s and $-1$s in an alternative manner starting with $1$. The index $l=1$ is used for the main anti-diagonal, for example
\begin{equation}\label{notation-K}
\scalemath{.9}{K_{4}(1)=\begin{pmatrix}
0&0&0&1\\
0&0&-1&0\\
0&1&0&0\\
-1&0&0&0\\
\end{pmatrix}}
\end{equation}
For $m<n$ we set ${K_{m\times n}(l)=\begin{pmatrix}
\0& K_m(l)
\end{pmatrix}}$ and for $m>n$: $K_{m\times n}(l)=\begin{pmatrix}
K_m(l)\\
\0
\end{pmatrix},$ where $l=1,.\ldots,\min \{ m,n \}$. 
For example
\begin{equation}\label{notation-K-2}
\scalemath{.9}{K_{4\times 5}(1)=\begin{pmatrix}
0&0&0&0&1\\
0&0&0&-1&0\\
0&0&1&0&0\\
0&-1&0&0&0\\
\end{pmatrix}},\quad \scalemath{.9}{K_{5\times 4}(4)=\begin{pmatrix}
0&0&0&0\\
0&0&0&0\\
0&0&0&0\\
0&0&0&1\\
0&0&0&0
\end{pmatrix}.}
\end{equation}
 In the case of complex eigenvalues we will use a similar notation $K^c_{2m\times 2n}(l),$\, ${l=1,\ldots,\min \{m,n\}}$ where  anti-diagonals are replaced by anti-diagonals made of $2\times 2$ blocks and $1,-1$ are replaced by $I_2,-I_2$ matrices. For example
\begin{equation}\label{notation-Kc}
\scalemath{.9}{K^c_{8\times 10}(2)=\begin{pmatrix}
\0&\0&\0&\0&\0\\
\0&\0&\0&\0&I_2\\
\0&\0&\0&-I_2&\0\\
\0&\0&I_2&\0&\0
\end{pmatrix}},\quad \scalemath{.9}{K^c_{10\times 8}(3)=\begin{pmatrix}
\0&\0&\0&\0\\
\0&\0&\0&\0\\
\0&\0&\0&I_2\\
\0&\0&-I_2&\0\\
\0&\0&\0&\0
\end{pmatrix}}.
\end{equation}
  We now present our next result.
\begin{theorem}\label{diagonal-form-D}
For the matrix $B$ of the form 
 \begin{equation}\label{B-in-jordan-form}
 \scalemath{.9}{
 {\rm diag}(\J(0),\underbrace{\J(\lambda_1),\J(-\lambda_1), \ldots\ldots}_{\mbox{\scalebox{.6}{real nonzero eigenvalues}}},\underbrace{\J(\pm \theta_1i),\ldots\ldots}_{\mbox{\scalebox{.6}{pure imaginary eigenvalues}}},\underbrace{\J(a_1\pm b_1i),\J(-(a_1\pm b_1i)),\ldots}_{\mbox{\scalebox{.6}{complex eigenvalues with $a\neq 0$}}}),}
\end{equation}
 the  symmetric matrix $D$ which makes $DB$ skew-symmetric is of the form
\begin{equation}\label{matrix-D}
D=\diag(\Dd(0),\Dd(\pm\lambda_1),\ldots,\Dd(\pm  \theta_1 i),\ldots,\ldots,\Dd(\pm(a_1\pm b_1i)),\ldots),
\end{equation}
where
\begin{enumerate}[(i)]
\item The matrix $\Dd(0)$   is of the form
\begin{equation}\label{D-for-eigenvalue-zero}
\begin{pmatrix}
D_{r}&D_{r,s_1}&D_{r,s_2}&\ldots&D_{r,s_k}\\
D^t_{r,s_1}&D_{s_1}&D_{s_1,s_2}&\ldots&D_{s_1,s_k}\\
D^t_{r,s_2}&D^t_{s_1,s_2}&\ddots&&\vdots\\
\vdots&\vdots&&\ddots&\vdots\\
D^t_{r,s_k}&D^t_{r,s_k}&\ldots&\ldots&D_{s_k}
\end{pmatrix}
\end{equation}
where
\begin{enumerate}[(1)]
\item $D_{r}$ is an arbitrary symmetric matrix,
\item all the elements of $D_{r,s_j}$ are zero except the last column which is an arbitrary $r$-dimensional column vector,
\item $D_{s_j}=\sum_{i=0}^kd_{s_j-2l}K_{s_j}(s_j-2l)$ where ${s_j-2k=1\,\, \mbox{or}\,\, 2}$, the matrix ${K_{s_j}(.)}$ is  defined in~\eqref{notation-K} and  $d_{s_j-2l}$ is an arbitrary number.
 For example
\begin{equation}\label{D4-D5}\scalemath{.8}{D_{4}=\begin{pmatrix}
0&0&0&0\\
0&0&0&d_{2}\\
0&0&-d_{2}&0\\
0&d_{2}&0&d_{4}
\end{pmatrix}},\quad\scalemath{.8}{D_{5}=\begin{pmatrix}
0&0&0&0&d_{1}\\
0&0&0&-d_{1}&0\\
0&0&d_{1}&0&d_{3}\\
0&-d_{1}&0&-d_{3}&0\\
d_{1}&0&d_{3}&0&d_{5}\\
\end{pmatrix}.}
\end{equation}
\item  $D_{s_i,s_j}=\sum_{l=1}^{\min \{s_i,s_j\}}d_l K_{s_i\times s_j}(l)$ where $d_l$ is an arbitrary number and $K_{s_i\times s_j}(.)$ is defined in~\eqref{notation-K-2}.
For example
$$
\scalemath{.8}{D_{4,5}=\begin{pmatrix}
0&0&0&0&d_{1}\\
0&0&0&-d_{1}&d_{2}\\
0&0&d_{1}&-d_{2}&d_{3}\\
0&-d_{1}&d_{2}&-d_{3}&d_{4}
\end{pmatrix},\quad D_{5,4}=\begin{pmatrix}
0&0&0&d_{1}\\
0&0&-d_{1}&d_{2}\\
0&d_{1}&-d_{2}&d_{3}\\
-d_{1}&d_{2}&-d_{3}&d_{4}\\
0&0&0&0
\end{pmatrix}.}
$$
\end{enumerate}

\item The matrix $\Dd(\pm\lambda_1)$  is of the form
\begin{equation}\label{D-for-eigenvalue-real-nonzero}
\scalemath{.9}{
\begin{pmatrix}
\0_{r^+\times r^+}&\0&\ldots&\0 &D_{r^+,r^-}&D_{r^+,s_1^-}&\ldots&D_{r^+,s_l^-}
\\
\0&\0_{s^+_1\times s^+_1}&\ldots&\0 &D^t_{r^-,s_1^+}&D_{s_1^+,s_1^-}&\ldots&D_{s_1^+,s_l^-}
\\
\vdots &\vdots&\ddots &\vdots &\vdots&\vdots&\ddots&\vdots
\\
\0&\0&\ldots&\0_{s^+_k\times s^+_k}&D^t_{r^-,s_k^+}&D_{s_k^+,s_1^-}&\ldots&D_{s_k^+,s_l^-}
\\
D^t_{r^+,r^-}&D_{r^-,s_1^+}&\ldots&D_{r^-,s_k^+}&\0_{r^-\times r^-}&\0&\ldots&\0
\\
D^t_{r^+,s_1^-}&D^t_{s_1^+,s_1^-}&\ldots&D^t_{s_k^+,s_1^-}&\0&\0_{s^-_1\times s^-_1}&\ldots&\0
\\
\vdots&\vdots&\ddots&\vdots&\vdots&\vdots&\ddots&\vdots
\\
D^t_{r^+,s_l^-}&D^t_{s_1^+,s_l^-}&\ldots&D^t_{s_k^+,s_l^-}&\0&\0&\ldots&\0_{s^-_l\times s^-_l}
\end{pmatrix}.}
\end{equation}
In this matrix the indices with positive superscript refer to the dimensions of the blocks in $\J(\lambda_1)$ and the indices with negative superscript to the dimensions of the blocks in $\J(-\lambda_1)$. Furthermore, 

\begin{enumerate}[(1)]
\item $D_{r^+,r^-}$ is an arbitrary matrix,
\item the matrices of type $D_{r^\ast,s^{-\ast}_j}$ where $\ast=+,-$  have all the elements equal zero except the last column which is an arbitrary $r^\ast$-dimensional column vector,
\item $D_{s_i^+,s_j^-}=\sum_{l=1}^{\min \{s_i^+,s_j^-\}}d_l K_{s_i\times s_j}(l)$ where $d_l$ is an arbitrary number and $K_{s_i\times s_j}(.)$ is defined in~\eqref{notation-K-2}. This is the same matrix as in item {\bf (i)}-$4$.
\end{enumerate}

\item The matrix $\Dd(\pm \theta_1 i)$  is of the form
\begin{equation}\label{D-for-eigenvalue-pure-imaginary}
\begin{pmatrix}
D_{2r}&D_{2r,2s_1}&D_{2r,2s_2}&\ldots&D_{2r,2s_k}\\
D^t_{2r,2s_1}&D_{2s_1}&D_{2s_1,2s_2}&\ldots&D_{2s_1,2s_k}\\
D^t_{2r,2s_2}&D^t_{2s_1,2s_2}&\ddots&&\vdots\\
\vdots&\vdots&&\ddots&\vdots\\
D^t_{2r,2s_k}&D^t_{2r,2s_k}&\ldots&\ldots&D_{2s_k}
\end{pmatrix}
\end{equation}
where
\begin{enumerate}[(1)]
\item $D_{2r}=\begin{pmatrix}
d_{11}I_2&D_{(\alpha_{12}\pm \beta_{12}i)}&\ldots&D_{(\alpha_{1r}\pm \beta_{1r}i)}\\
D_{(\alpha_{12}\pm \beta_{12}i)}&d_{22}I_2&\ldots&D_{(\alpha_{2r}\pm \beta_{2r}i)}\\
\vdots&\vdots\vdots\vdots&\ddots&\vdots\\
D_{(\alpha_{1r}\pm \beta_{1r}i)}&D_{(\alpha_{2r}\pm \beta_{2r}i)}&\ldots&d_{rr}I_2
\end{pmatrix}$ wherein $\scalemath{.9}{D_{(\alpha_{ij}\pm \beta_{ij}i)}=\begin{pmatrix}\alpha_{ij}&\beta_{ij}\\-\beta_{ij}&\alpha_{ij}\end{pmatrix}},$ 

\vspace{.5cm}
\item $D_{2r,2s_i}=\begin{pmatrix}
\0&\ldots&\0&D_{(\alpha_1\pm \beta_1i)}\\\0&\ldots&\0&D_{(\alpha_2\pm \beta_2i)}\\
\vdots&\vdots&\vdots&\vdots\\
\0&\ldots&\0&D_{(\alpha_r\pm \beta_ri)}
\end{pmatrix}_{r\times s_j},$

\item $D_{2s_j}=D_0+D_1$ where
\begin{align*}
&D_0=\sum_{l=0}^{k_0}\alpha_{s_j-2l}K^c_{2s_j}(s_j-2l),\\
&D_1=\sum_{l=0}^{k_1}D_{(\pm\beta_{(s_j-(2l+1))}i)}K^c_{2s_j}(s_j-(2l+1)),
\end{align*}
 and $K^c_{2s_j}(.)$ is defined in~\eqref{notation-Kc}. Depending on $s_j$ the indices $(s_j-2k_0),{(s_j-(2k_1+1))}$ go down to either $1$
 or  $2$. For example
\begin{equation}\label{D8-D10-case(iii)}
\scalemath{.7}{D_{8}=\begin{pmatrix}
\0&\0&\0&D_{\pm(\beta_1 i)}\\
\0&\0&-D_{(\pm\beta_1 i)}&\alpha_2I_2\\
\0&D_{(\pm\beta_1 i)}&-\alpha_2I_2&D_{(\pm\beta_3 i)}\\
-D_{(\pm\beta_1 i)}&\alpha_2I_2&-D_{(\pm\beta_3 i)}&d_4 I_2
\end{pmatrix}},\quad\scalemath{.7}{D_{10}=\begin{pmatrix}
\0&\0&\0&\0&\alpha_1I_2\\
\0&\0&\0&-\alpha_1I_2&D_{(\pm\beta_2 i)}\\
\0&\0&\alpha_1I_2&-D_{(\pm\beta_2 i)}&\alpha_3I_2\\
\0&-\alpha_1I_2&D_{(\pm\beta_2 i)}&-\alpha_3I_2&D_{(\pm\beta_4 i)}\\
\alpha_1I_2&-D_{(\pm\beta_2 i)}&\alpha_3I_2&-D_{(\pm\beta_4 i)}&\alpha_5I_2\\
\end{pmatrix}.}
\end{equation}
\item  Similar to item ${\bf i}$-$4$
 $$D_{2s_i,2s_j}=\sum_{l=1}^{\min\{s_i,s_j\}} D_{(\alpha_l\pm\beta_l i)}K^c_{2s_i\times 2s_j}(l)$$ where $K^c_{2s_i\times 2s_j}(.)$ is defined in~\eqref{notation-Kc}.
For example
$$
\scalemath{.8}{D_{8,10}=\begin{pmatrix}
\0&\0&\0&\0&D_{(\alpha_1\pm\beta_1i)}\\
\0&\0&\0&-D_{(\alpha_1\pm\beta_1i)}&D_{(\alpha_2\pm\beta_2i)}\\
\0&\0&D_{(\alpha_1\pm\beta_1i)}&-D_{(\alpha_2\pm\beta_2i)}&D_{(\alpha_3\pm\beta_3i)}\\
\0&-D_{(\alpha_1\pm\beta_1i)}&D_{(\alpha_2\pm\beta_2i)}&-D_{(\alpha_3\pm\beta_3i)}&D_{(\alpha_4\pm\beta_4i)}
\end{pmatrix}.}
$$
\end{enumerate}
\item The matrix $\Dd_{2s^+,2s^-}(\pm(a\pm bi))$  is of the form
\begin{equation}\label{D-for-eigenvalue-complex}
\scalemath{.9}{
\begin{pmatrix}
\0_{2r^+\times 2r^+}&\0&\ldots&\0 &D_{2r^+,2r^-}&D_{2r^+,2s_1^-}&\ldots&D_{2r^+,2s_k^-}
\\
\0&\0_{2s^+_1\times 2s^+_1}&\ldots&\0 &D^t_{2r^-,2s_1^+}&D_{2s_1^+,2s_1^-}&\ldots&D_{2s_1^+,2s_k^-}
\\
\vdots &\vdots&\ddots &\vdots &\vdots&\vdots&\ddots&\vdots
\\
\0&\0&\ldots&\0_{2s^+_k\times 2s^+_k}&D^t_{2r^-,2s_k^+}&D_{2s_k^+,2s_1^-}&\ldots&D_{2s_k^+,2s_k^-}
\\
D^t_{2r^+,2r^-}&D_{2r^-,2s_1^+}&\ldots&D_{2r^-,2s_k^+}&\0_{2r^-\times 2r^-}&\0&\ldots&\0
\\
D^t_{2s_1^-,2r^+}&D^t_{2s_1^+,2s_1^-}&\ldots&D^t_{2s_k^+,2s_1^-}&\0&\0_{2s^-_1\times 2s^-_1}&\ldots&\0
\\
\vdots&\vdots&\ddots&\vdots&\vdots&\vdots&\ddots&\vdots
\\
D^t_{2s_k^-,2r^+}&D^t_{2s_1^+,2s_k^-}&\ldots&D^t_{2s_k^+,2s_k^-}&\0&\0&\ldots&\0_{2s^-_k\times 2s^-_k}
\end{pmatrix},}
\end{equation}
where

\begin{enumerate}[(1)]
\item $D_{2r^+,2r^-}=\begin{pmatrix}
D_{(\alpha_{11}\pm \beta_{11}i)}&D_{(\alpha_{12}\pm \beta_{12}i)}&\ldots&D_{(\alpha_{1r^-}\pm \beta_{1r^-}i)}\\
D_{(\alpha_{12}\pm \beta_{12}i)}&D_{(\alpha_{22}\pm \beta_{22}i)}&\ldots&D_{(\alpha_{2r^-}\pm \beta_{2r^-}i)}\\
\vdots&\vdots\vdots\vdots&\ddots&\vdots\\
D_{(\alpha_{r^+1}\pm \beta_{r^+1}i)}&D_{(\alpha_{r^+2}\pm \beta_{r^+2}i)}&\ldots&D_{(\alpha_{r^+r^-}\pm \beta_{r^+r^-}i)}
\end{pmatrix},$
\item $D_{2r^\ast,2s^{-\ast}_j}=\begin{pmatrix}
\0&\ldots&\0&D_{(\alpha_1\pm \beta_1i)}\\\0&\ldots&\0&D_{(\alpha_2\pm \beta_2i)}\\
\vdots&\vdots&\vdots&\vdots\\
\0&\ldots&\0&D_{(\alpha_{2r^\ast}\pm \beta_{2r^\ast} i)}
\end{pmatrix}_{2r^\ast \times 2s^{-\ast}_j}$ where $\ast=+,-,$
\item  $D_{2s^+_i,2s^-_j}$ is as defined in item {\bf (iii)}-$4$.
\end{enumerate}
\end{enumerate}
\end{theorem}
\begin{proof}
The proof is similar to the proof  of~\cite[Theorem (9.1.1)]{MR2228089}. We first prove that $D$ has the diagonal form~\eqref{matrix-D}. In order to do so, we need to discuss the following three cases.
\begin{itemize}
\item[1)] $B=\diag(J_{s_1}(\lambda_1),J_{s_2}(\lambda_2))$ where $\lambda_1+\lambda_2\neq 0$. We write the matrix $D$ in the block form 
 $$D=\begin{pmatrix}Z_{s_1}&Z\\Z^t&Z_{s_2}\end{pmatrix}.$$
The matrix $DB$ being skew-symmetric implies
\begin{equation}\label{equ:diagonal-form0}
Z J_{s_2}(\lambda_2)=-J^t_{s_1}(\lambda_1) Z.
\end{equation}
 Rewriting this equation, we have
\begin{equation}\label{diagonal-form-equation}
(\lambda_2+\lambda_1) Z=-(H^t_{s_1} Z+ Z H_{s_2}),
\end{equation}
where 
\begin{equation}\label{H-m}
H_{s_i}=\begin{pmatrix}0&1&0&...&0\\0&0&1&...&0\\\vdots&\vdots&\ddots&\ddots&\vdots\\0&0&0&\ddots&1\\0&0&0&...&0\end{pmatrix},
\end{equation}
is the nilpotent upper shift matrix. Now, we multiply both sides of~\eqref{diagonal-form-equation} by ${(\lambda_2+\lambda_1)}.$ Using Equality~\eqref{diagonal-form-equation} we get
\[ (\lambda_2+\lambda_1)^2 Z=(H^t_{s_1})^2+2H^t_{s_1}Z H_{s_2}+Z (H_{s_2})^2.\]
Repeating this process, we obtain that for every $p=1,2,3,..$
\begin{equation}\label{diagonal-2}(\lambda_2+\lambda_1)^p Z=\sum_{q=1}^p\begin{pmatrix}p\\q\end{pmatrix}(H^t)^q_{s_1}Z H^{p-q}_{s_2}.\end{equation}
For $p$ large enough the right hand side of Equation~\eqref{diagonal-2} vanishes. Since ${\lambda_2+\lambda_1\neq0}$, the off-diagonal bloc $Z$ vanishes.
\item[2)] $B=\diag(J_{2s_1}(a_1\pm b_1i),J_{2s_2}(a_2\pm b_2i)$ where $(a_1\pm b_1i)\neq - (a_2\pm b_2i)$. As in item $(1)$, the matrix  $DB$ being skew-symmetric implies 
\begin{equation}\label{eq:diagonal-2}Z J_{2s_2}(a_2\pm b_2i)=-J^t_{2s_1}(a_1\pm b_1i) Z.\end{equation}
 We divide $Z$ into $2\times 2$ blocks $Z_{ij},\, i=1,...,s_1$ and $j=1,...,s_2$.  An strait forward calculation shows that the equation $$Z_{ij}B_{(a_2+b_2i)}=-B_{(a_2+b_2i)}Z_{ij}$$
 has nontrivial solution for $Z_{ij}$ if and only if  $(a_1+a_2)= 0$ and $(b_1^2-b_2^2)= 0$ i.e., $(a_1\pm b_1i) = - (a_2\pm b_2i)$ which is not the case here. Now, Equation~\eqref{eq:diagonal-2} implies that:
 \begin{align*}
& Z_{11}B_{(a_2+b_2i)}=-B_{(a_2+b_2i)}Z_{11},\\&Z_{11}+Z_{12}B_{(a_2+b_2i)}=-B_{(a_2+b_2i)}Z_{12},\\&\vdots \\&Z_{1(s_j-1)}+Z_{1s_j}B_{(a_2+b_2i)}=-B_{(a_2+b_2i)}Z_{1s_j}.\end{align*}
which yields $Z_{11}=Z{12}=\ldots=Z_{1s_j}=0$.
\item[$3)$] $B=\diag(J_{s_1}(\lambda_1),J_{2s_2}(a_2+b_2i))$ where $b_2\neq0$. Similarly, the matrix  $DB$ being skew-symmetric implies 
\begin{equation}\label{eq:diagonal-3}
Z J_{2s_2}(a_2+b_2i)=-J^t_{s_1}(\lambda_1) Z.
\end{equation}
We divide $Z$ into $s_1\times 2$ blocks $Z^j,\,j=1,..,s_2.$ It is very easy to see that the equation 
$$Z^jB_{(a_2+b_2i)}=-J^t_{s_1}(\lambda_1) Z^j,$$ has nontrivial solution if and only if $(a_2-\lambda_1)^2+b_2^2=0$ which is not the case here. 
Equation~\eqref{eq:diagonal-3} implies 
\begin{align*}
&Z^1B_{(a_2+b_2i)}=-J^t_{s_1}(\lambda_1) Z^1,\\
&Z^1+Z^2B_{(a_2+b_2i)}=-J^t_{s_1}(\lambda_1) Z^2,\\
&\vdots\\
&Z^{s_j-1}+Z^{s_j}B_{(a_2+b_2i)}=-J^t_{s_1}(\lambda_1) Z^{s_j},\\
\end{align*}
 \end{itemize}
 which, again, yields  $Z=0$. The rest of the proof requires some cumbersome but strait forward calculations.  For any item we provide an outline of the proof without going into detailed calculations. 
\begin{itemize}
\item[{\bf i)}]
The block $D_r$  is an  arbitrary symmetric matrix since it is associated to the zero block of $B$. The symmetric diagonal blocks $D_{s_i},\,i=1,...,k$ should satisfy 
${H^t_{s_i}D_{s_i}=-D_{s_i}H_{s_i}}$ where $H_{s_i}$ is the shift matrix defined in~\eqref{H-m}. 
The product $H^t_{s_i}D_{s_i}$ is obtained from $D_{s_i}$ shifting all the rows one place downward and filling the first row with zeros. Similarly, $D_{s_i}H_{s_i}$ is obtained from $D_{s_i}$ shifting all the columns one place to the right and filling the first column with zeros. 
Having this in mind, it is quite strait forward to verify that the matrix $D_{s_i}$ has the form described in the lemma.  The off-diagonal blocks $D_{s_i,s_j}$ should satisfy ${H^t_{s_i}D_{s_i}=-D_{s_j}H_{s_j}}$.  Again, applying the downward and rightward shifts one can verify that $D_{s_i,s_j}$ has the form stated in the lemma. Note that the matrix $D_{s_i,s_j}$ is not necessarily quadratic nor symmetric. A similar line of reasoning applies to $D_{r,s_i}$.
\item[{\bf ii)}] In this case, the required equation for $DB$ to be skew-symmetric is 
$$(\lambda_2+\lambda_1) Z=-(H^t_{s_1} Z+ Z H_{s_2}),$$
where $\lambda_1,\lambda_2=\pm\lambda$. When $\lambda_1=\lambda_2$,  the same reasoning, used in the beginning of the proof, applies and the block $Z$ should be zero. When $\lambda_2=-\lambda_1$ exploiting the shifting properties of $H_{s_i}$ shows that 
$D_{r^\ast,s^{-\ast}_j}$ and $D_{s^+_i,s^-_j}$ have the form given in the lemma. For $D_{r^+,r^-}$ the dimension of the Jordan blocks (associated to non-degenerate eigenvalues) involved are one i.e., $H_{s_i}=0$ so it could be arbitrary. 
\item[{\bf iii)}] We divide the matrix $\Dd(\pm \theta i)$ into $2\times 2$ blocks $Z_{ij}$.  As for the block $D_{2r}$, the matrix $DB$ being skew-symmetric implies that these blocks should satisfy 
$$Z_{ij}\begin{pmatrix}0&b\\-b&0\end{pmatrix}=-\begin{pmatrix}0&-b\\b&0\end{pmatrix}Z_{ij}.$$
  Solving this equation, one gets the symmetric blocks $Z_{ii}=d_{ii} I $ and the blocks $B_{(\alpha_{ij}+\beta_{ij}i)},\, i<j$ as claimed in the lemma. 
 The first "lines" ( in the two by two block form)  of the blocks $D_{2s_i}$, $D_{2s_i,2s_j}$ should satisfy
\begin{align*}
& Z_{11}\begin{pmatrix}0&b\\-b&0\end{pmatrix}=-\begin{pmatrix}0&b\\-b&0\end{pmatrix}Z_{11},\\&Z_{11}+Z_{12}\begin{pmatrix}0&b\\-b&0\end{pmatrix}=-\begin{pmatrix}0&b\\-b&0\end{pmatrix}Z_{12},\\&\vdots \\&Z_{1(s_j-1)}+Z_{1s_j}\begin{pmatrix}0&b\\-b&0\end{pmatrix}=-\begin{pmatrix}0&b\\-b&0\end{pmatrix}Z_{1s_j}.\end{align*}
Again solving these equations for symmetric matrix $D_{2s_i}$ and continuing on with other lines give the form of the block $D_{2s_i}$. The same applies to the arbitrary block $D_{2s_i,2s_j}$ and similar process works for $D_{2r,2s_i}$. 
\item[{\bf iv)}] The proof of this item is similar to item $(3)$. 
\end{itemize}
\end{proof}

The following comments are some the outcomes of our analysis.
\begin{remark} \label{remark-main-results-linear}
For a given linear vector field $X=Bu$  we have: 
\begin{itemize}
\item[1)] If the matrix $B$ has at least one pair of nonzero positive-negative eigenvalues, real or complex, then  $X$ has a Hamiltonian description with the quadratic Hamiltonian $H(u)=u^t D u$ w.r.t a big-isotropic structure.  Furthermore, if the matrix $B$ is non-singular, then the underlying structure is a presymplectic one.
\item[2)]  In the case of eigenvalue zero, presence of a three dimensional Jordan block guaranties a Hamiltonian description w.r.t. a big-isotropic structure. 
\item[3)] Let $d_\xi$ be one of the free variables of $D$, we define $D_\xi$ to be the matrix obtained from $D$ putting $d_\xi=1$ and all the other free variables equal to zero. Clearly,  the function $F_\xi=u^tD_\xi u$ is a constant of motion. The set $\{F_\xi\}_\xi$, where $\xi$ runs over all the free variables of $D$, contains all the first integrals of $X$ achievable by our algorithm. We will use this set of first integrals to show that the vector field $X$ is completely integrable in the sense  of Zung \cite{zung-action-angle}. 
\item[4)] The free variables of $D$ yield alternative Hamiltonian description for $X$.  
\end{itemize}
\end{remark}
We continue analysing the outcomes. 
\subsection*{Eigenvalue zero}
The following lemma proves the existence of a class of linear vector fields which can only have  Hamiltonian description  w.r.t. to a big-isotropic structure which is not Dirac. 
\begin{lemma}\label{only-proper-big-iso}
If there is no possibility for the matrix $\Dd(0)$, in item {\bf (i)} of Theorem~\ref{diagonal-form-D}, to be non-singular, then $\ker \Dd(0)\cap \ker \J(0)\neq 0$. It can be chosen non-singular if and only if it contains only  odd dimensional\footnote{Every non-degenerate eigenvalue is considered as a one dimensional Jordan block.} Jordan blocks $J_{2k+1}(0)$  or pairs of $(J_{2l},J_{2l})$. 
\end{lemma}
\begin{proof}
The first block, $D_r$, of $\Dd(0)$  is an arbitrary symmetric matrix, so it can be chosen to be non-singular. In the rest of the proof we always assume that $D_r$ is non-singular.   As we mentioned before, without losing generality, we may assume that $s_1\leq s_2\leq \ldots\leq s_k$. Our proof goes by induction on the number of degenerate Jordan blocks. For $k=1$ if  $J_{s_1}(0)$ is even dimensional, then the anti-diagonal and all the upper anti-diagonals of $D_{s_1}$ are zero -- see~\eqref{D4-D5} for the case of $s_1=4$. This means that $\Dd(0)$ is always singular. Furthermore, since its  $(r+1)^{\rm th}$ column is zero, we have
$$e_{r+1}\in\ker \Dd(0)\cap \ker \J(0)\neq 0.$$ For example, when $B$ has $r$ non-degenerate zero eigenvalues and a Jordan block of dimension $4$ associated to a degenerate eigenvalue zero, we have 
\[\Dd(0)=\begin{pmatrix}D_r&0&0&0&d_{1\times r}\\
0&0&0&0&0\\
0&0&0&0&d_{2}\\
0&0&0&-d_{2}&0\\
d^t_{1\times r}&0&d_{2}&0&d_{4}\\
\end{pmatrix}.\]
  If $s_1$ is  odd dimensional, its anti-diagonal can be chosen to be non-zero and the rest of its elements to be zero -- see~\eqref{D4-D5} for the case of $s_1=5$. This means  the matrix $D_{s_1}$ and, consequently, the matrix $\Dd(0)$ can be chosen to be non-singular. 
  
  Now assume that the statement of the lemma holds when $\J(0)$ has $k-1$ degenerate Jordan blocks. We consider two cases. 
\begin{itemize}
\item[$1)$] If $s_k$ is odd dimensional, we chose $D_{s_k}$ to be non-singular. Then write $\Dd(0)$ in the  form
\[\Dd(0)=\begin{pmatrix}
D_1&D_2\\D^t_2&D_{s_k}
\end{pmatrix}.\]
For the non-singular matrix $T=\begin{pmatrix}
I_{(s-s_k)}&\0\\-D^{-1}_{s_k}D_2^t&I_{s_k}
\end{pmatrix}$ we have
\begin{equation}\label{change-D}T^t\Dd(0) T=\begin{pmatrix}
D_1-D_2D^{-1}_{s_k}D_2^t&\0\\\0&D_{s_k}
\end{pmatrix}.\end{equation}
Now, we show that 
\begin{equation}\label{change-J}
T^{-1}\J(0) T=\J(0).
\end{equation}
The inverse of $T$ is $T^{-1}=\begin{pmatrix}
I_{(s-s_k)}&\0\\D^{-1}_{s_k}D_2^t&I_{s_k}
\end{pmatrix}$. 
So, considering $\J(0)$ in the form $\begin{pmatrix}
J_1&\0\\\0&J_{s_k}
\end{pmatrix},$ we only need to show that 
\begin{equation}\label{eq:1}D_{s_k}^{-1}D_2^{t}J_1-J_{s_k}D^{-1}_{s_k}D_2^t=0.\end{equation}
Since $\Dd(0)\J(0)$ is skew-symmetric  we have 
\[D_2^{t}J_1=-(D_2J_{s_k})^t\quad \mbox{and}\quad  J^t_{s_k}D_{s_k}+D_{s_k}J_{s_k}=0.\] 
Exploiting these two equalities  yields \eqref{eq:1}. 
Equations \eqref{change-D} and \eqref{change-J} together with Remark~\ref{normalizing-to-jordan} show that doing a change of variable by $T$ we can restrict ourself to the matrix $J_1$ which has $(k-1)$ degenerate Jordan blocks. This in turn proves Lemma by induction.
\item[2)] If $s_k$ is even, then $D_{s_k}$ is singular for sure. Now if $s_{k-1}\neq s_{k}$, then the column $s-s_k+1$ of $\Dd(0)$ is zero and $$e_{(s-s_k+1)}\in\ker \Dd(0)\cap \ker \J(0)\neq 0.$$

 For the case where $s_k=s_{k-1}$,  we write $\Dd(0)$ in the form
\[\Dd(0)=\begin{pmatrix}
D_1&D_2\\D^t_2&D_4
\end{pmatrix},\]
where $D_4=\begin{pmatrix}D_{s_{k}}&D_{s_{k},s_k}\\D^t_{s_{k},s_k}&D_{s_k}\end{pmatrix}$ and $D_{s_{k},s_k}=\sum_{l=1}^{s_k}d_l K_{s_{k}\times s_k}(l)$, as in item {\bf (i)}-$4$ of Theorem~\ref{diagonal-form-D}. The matrix $D_4$ can be chosen to be non-singular. Simply, set $D_{s_k}=\0$ and for the matrix $D_{s_{k},s_k}$ set ${d_l=0,\,\, \forall l\neq 1}$ and $d_1=1$.  Now, we can repeat a similar change of variable as in item $(1)$ and restrict our self to a matrix with $(k-2)$ Jordan blocks which proves Lemma by induction.   
\end{itemize}
 \end{proof}

\begin{corollary}\label{cor:only-proper-big-iso}
Let $B$ be a matrix which has a single even dimensional Jordan block $J_{2k}(0)$ i.e.,  any other Jordan blocks $J_s(0)$ of $B$ has dimension different than $2k$, then the  linear vector fields $X=Bu$ has Hamiltonian description only w.r.t. a big-isotropic structure which is not Dirac.
\end{corollary}
\begin{proof}
Since the fiber of any Dirac structure $L$ (of any type not only the ones discussed here) at a given point $u_0$ is of the form
\[L(u_0)=\{(B_0z,D_0^tz) |  \forall z\in\mathbb{R}^m\},\] 
for some constant matrices $B_0$ and $D_0$, applying Lemma~\ref{only-proper-big-iso} the structure can not be Dirac at $u_0$ which means  the structure is not Dirac. 
\end{proof}
 
The Hamiltonian description presented in~\cite{MR1233228}, repeated in~\cite[Theorem 4.2]{geo-from-dyn}, requires the zero eigenvalues to have even multiplicity.   We apply our method to the following example from~\cite{MR1233228} showing that our approach in addition to removing this requirement, detects casimirs as well.
\begin{example}\label{example:zero-eigenvalue}

Let $X=\tilde{B}u$ where $\tilde{B}=\diag(G,G,...,G)$ with 
$$G=\begin{pmatrix}
0&0&1&0\\0&0&0&1\\0&0&0&0\\0&0&0&0
\end{pmatrix}.$$
 It is shown in~\cite{MR1233228} that $X$ is Hamiltonian w.r.t. the symplectic structure
\[\Omega=\begin{pmatrix}
\0&\ldots&\ldots&\ldots&I_2\\
\ldots&\ldots&\ldots&-I_2&\0\\
\ldots&\ldots&I_2&\0&\ldots\\
\ldots&-I_2&\0&\ldots&\ldots\\
\ldots&\ldots&\ldots&\ldots&\ldots\\
\ldots&\ldots&\ldots&\ldots&\ldots\\
\end{pmatrix}\]
and  the Hamiltonian function defined by the symmetric matrix $\Omega \tilde{B}$. The matrix $\tilde{B}$ is not in the Jordan normal form. To do a comparison, we consider $\tilde{B}=G$ i.e., with only one block. The Hamiltonian given by $\Omega \tilde{B}$  is then $H_1(u)=-\frac{1}{2}(u_3^2+u_4^2)$. To apply our approach, we do the change of variable $u=Tv$ with $T=\left(
\begin{array}{cccc}
 0 & 0 & 1 & 0 \\
 1 & 0 & 0 & 0 \\
 0 & 0 & 0 & 1 \\
 0 & 1 & 0 & 0 \\
\end{array}
\right)$. Then
\[ X(v)= \underbrace{(T^{-1}.G.T)}_{=B}v=\left(\begin{array}{cccc}
 0 & 1 & 0 & 0 \\
 0 & 0 & 0 & 0 \\
 0 & 0 & 0 & 1 \\
 0 & 0 & 0 & 0 \\
\end{array}\right)v,\]
i.e., $B=\diag(J_2(0),J_2(0))$. By item {\bf (i)} of Theorem~\ref{diagonal-form-D} we have:
\[D=\begin{pmatrix}
0&0&0&d_{14}\\
0&d_{22}&-d_{14}&d_{24}\\
0&-d_{14}&0&0\\
d_{14}&d_{24}&0&d_{44}
\end{pmatrix}.
\]
If $d_{14}\neq 0$, then $D$ is non-singular and $X(v)$ is Hamiltonian w. r. t. the Poisson structure $BD^{-1}=\left(
\begin{array}{cccc}
 0 & 0 & -\frac{1}{d_{14}} & 0 \\
 0 & 0 & 0 & 0 \\
 \frac{1}{d_{14}} & 0 & 0 & 0 \\
 0 & 0 & 0 & 0 \\
\end{array}
\right)$ and the Hamiltonian function 
$$H(v)=\frac{1}{2}(v^tDv)=\frac{1}{2}(d_{22}v_2^2+d_{44}v_4^2+2d_{14}(v_1v_4-v_2v_3)).$$ 
Linear Casimirs given by the elements of  $\ker B$ are $v_2 (=u_4)$ and $v_4 (=u_3)$. So, the function $H_1(u)$ is actually a Casimir. In~\cite{MR1233228}, the vector field $\tilde{B}u$ is paired with a Casimir trough a symplectic structure. As mentioned there, this is only possible when zero eigenvalue has even multiplicity and its Jordan blocks have the same structure. Our approach works for any dimension and any type of Jordan blocks, detecting Casimirs. Note that the vector field $X$ is Hamiltonian w.r.t. the Poisson structure $BD^{-1}$ and the Hamiltonian $H^\prime(v)=d_{14}(v_1v_4-v_2v_3)$. 
\end{example}

\subsection*{Non-zero Eigenvalues}

If $B$ is non-singular, then the underlying structure for the Hamiltonian description  is the presymplectic structure $\omega^\sharp=D^tB^{-1}$. If the dimension of $B$ is odd, then it is not possible to have symplectic structure. In the case of even dimension, it is clear that if all eigenvalues of $B$ come in pairs $\pm\lambda_j$ or quadruples $\pm(a_j\pm b_ji)$ and their Jordan blocks have the same structure, then the underlying structure could be chosen to be symplectic. Moreover, it is well-known that  for a Hamiltonian vector field $X$, in the context of symplectic geometry,  the Jordan blocks belonging to a real or complex eigenvalue $\lambda$ have the same structure as the Jordan blocks belonging to $-\lambda$, see~\cite[page 453]{MR1233228}. The following example shows that our approach can handle vector fields that have Jordan blocks with different structures associated to a pair of positive negative eigenvalues. In addition, we will see that we are able to detect isotropic vector fields which in some sense are the analogues of Casimirs. 

\begin{example}\label{example:nonzero-eigenvalue}
Let 
\[X=x\frac{\partial}{\partial x}+y\frac{\partial}{\partial y}-\frac{1}{2}(w+z)\frac{\partial}{\partial z}+\frac{1}{2}(z-3w)\frac{\partial}{\partial w}.\]
Diagonalizing the representing matrix we have
\begin{equation}\label{not-symplectic}
X(v)=\left(
\begin{array}{cccc}
 1 & 0 & 0 & 0 \\
 0 & 1 & 0 & 0 \\
 0 & 0 & -1 & 1 \\
 0 & 0 & 0 & -1 \\
\end{array}
\right)v.
\end{equation}
As is mentioned in~~\cite{MR1233228}, since the Jordan blocks of the eigenvalues $\pm 1$ have different structures, the vector field $X$ can not have a Hamiltonian description with a symplectic structure as underlying structure. In our approach, item ${\bf (ii)}$ of Theorem~\ref{diagonal-form-D} provides 
\[D=\begin{pmatrix}
0&0&0&d_{14}\\
0&0&0&d_{24}\\
0&0&0&0\\
d_{14}&d_{24}&0&0
\end{pmatrix},\]
and the vector field  $X$ is Hamiltonian w.r.t. the presymplectic structure
$$DB^{-1}=\left(
\begin{array}{cccc}
 0 & 0 & 0 & -d_{14} \\
 0 & 0 & 0 & -d_{24} \\
 0 & 0 & 0 & 0 \\
 d_{14} & d_{24} & 0 & 0 \\
\end{array}
\right)$$
with the Hamiltonian function $H(u)=\frac{1}{2}u^tDu$. 
The kernel of $D$ is generated by $$\xi_1=\begin{pmatrix}0\\0\\1\\0\end{pmatrix},\quad \xi_2=\begin{pmatrix}
-\frac{d_{24}}{d_{14}}\\1\\0\\0
\end{pmatrix}\,\mbox{ i.e.,}$$
\[X_{\xi_1}=(0,0,-1,0)\quad\mbox{and}\quad X_{\xi_2}=(-\frac{d_{24}}{d_{14}},1),\]
are isotropic vector fields. This means that any linear combination ${X+c_1 X_{\xi_1}+c_2 X_{\xi_2}}$ is also Hamiltonian with the same Hamiltonian function $H$. 
\end{example}
\section{Integrability of Linear Systems}\label{sec:integrability}
In this Section, we discuss integrability of linear Hamiltonian vector fields. We provide a brief introduction to the concept of completely integrable systems in the sense of  Zung \cite{zung-action-angle}.  Following the terminology used by Zung, we omit  the word "completely". 

\begin{defn}\label{def:integrable-general}
Let $M$ be a $m$-dimensional manifold and $p\geq 1\,,q \geq 0$ such that $p+q = m$. A $m$ tuple  $(X_1,\ldots,X_p,F_1,\ldots,F_q)$, constituted of the vector fields  $X_i,\,i=1,\ldots,p$ and the smooth functions $F_j,\,j=1,\ldots,q$, is called {\it  an integrable system} of type $(p, q)$ on $M$ if it satisfies the following conditions:
\begin{itemize}
\item[(i)] $[X_i,X_j] = 0\quad\forall i,j = 1,...,p,$
\item[(ii)] $X_i(F_j)=0\quad \forall i\leq p,\, j\leq q,$
\item[(iii)] $X_1\wedge\ldots\wedge X_p\neq 0$ and $dF_1\wedge\ldots\wedge dF_q\neq 0$ almost everywhere on $M.$
\end{itemize}
A  vector field $X$ on a manifold $M$ is called  integrable if there exists an integrable system $(X_1,\ldots,X_p,F_1,\ldots,F_q)$ of some type $(p,q)$ on $M$ with $X_1 = X$.
\end{defn}

 The system defined above is called {\it regular } on a level set $N$ of first integrals, i.e., a level set of the map $(F_1,\ldots, F_q):M\to\mathbb{R}^q$, if  condition $(iii)$ holds everywhere on $N$.  An integrable system in the sense of Definition~\ref{def:integrable-general}  has {\it action-angle variables} (also known as {\it Liouville system of coordinates}) around any compact regular level set $N$, see \cite[Theorem 2.1]{zung-action-angle}. 

Definition~\ref{def:integrable-general} only considers commuting flows and first integrals without any underlying geometric structure, for that reason it is also called {\it non-Hamiltonian integrability}. The following definition considers an underlying geometric structure as well. 

\begin{defn}\label{def:integrable-hamiltonian}
An  integrable system $(X_1,\ldots,X_p,F_1,\ldots,F_q)$ is called {\it Hamiltonian integrable} on a manifold $M$ equipped with a big-isotropic  structures $L$ if there exist $H_1,\ldots,H_q\in C^\infty$ such that  $(X_i,d H_i)\in L,\,\, i=1,\ldots,q.$ 
A vector field $X$  is called {\it Hamiltonian integrable} if there exists a Hamiltonian integrable system $(X_1,\ldots,X_p,F_1,\ldots, F_q)$ of some type $(p, q)$ with $X_1=X$.
\end{defn}
Existence of action-angle variables for a Hamiltonian integrable system  having a presymplectic, Poisson or Dirac structure as underlying geometry is  proven in \cite{zung-action-angle}.  We are not aware of any work on the action-angle variable in the case of a big-isotropic structure which is not Dirac, but we believe that the results of \cite{zung-action-angle} can easily be extended for big-isotropic structures which are not Dirac. We now state and prove the main result of this section. 
\begin{theorem}\label{thm:integrability-main}
Every linear Hamiltonian system, $X=Bu$, is Hamiltonian integrable in the sense of Definition~\ref{def:integrable-hamiltonian}.
\end{theorem}
\begin{proof}
 As in Section~\ref{sec:linear-vector-fields}, by virtue of Remark~\ref{normalizing-to-jordan}, we consider the matrix $B$ in the form~\eqref{B-in-jordan-form}. The proof is  divided in several parts and  each part deals with an specific diagonal block of $B$.   It is clear that proving the theorem for each diagonal block separately would yield the proof for $B$. In each part we proceed in the following way:  

 Assume that the dimension of the diagonal block we are dealing with is $m^\prime$.  As mentioned in item $3$ of Remark~\ref{remark-main-results-linear}, every free variable $d_\xi$ in the matrix $D$, obtained by Theorem~\ref{diagonal-form-D}, yields a first integral.  We will pick $q^\prime$ of these integrals which are independent almost everywhere. These first integrals will be denoted  by $F_{\xi_j}=u^tD_{\xi_j}u,\, j=1,\ldots,q^\prime$ where $D_{\xi_j}$ is the matrix obtained from the  matrix $D$ putting $d_{\xi_j}=1$ and all the  other free variables equal to zero.  
 
 Then we introduce the  $p^\prime=m^\prime-q^\prime$ commuting vector fields $X_i$ that preserve these first integrals. We need the first vector field $X_1$ to be the vector field at hand therefore, we set $C_1$ the block part we are dealing with and $X_1=C_1u$. The rest of the commuting flows will be linear as well i.e., we present matrices $C_2,\ldots,C_{p^\prime}$ such that $X_i=C_iu,\,i=2,...,p^\prime$. The introduced matrices satisfy the following list of conditions
 \begin{lst}
 \label{list-1}
\begin{itemize}
\item[(1)] $X_1\wedge\ldots\wedge X_p\neq 0$ almost everywhere on $\mathbb{R}^m,$ 
\item[(2)]\label{item:condition2} $[C_i,C_j]=C_jC_i-C_iC_j=0,\quad\forall i,j=1,\ldots,p$,
\item[(3)] $D_{\xi_j}C_i$ is skew-symmetric for every $i=1,\ldots,p$ and $j=1,..,q$.
\end{itemize}
\end{lst}
Consequently, the vector field  $X_1$ is Hamiltonian  integrable w.r.t. a big-isotropic structure which will be described in each part.  
Note that 
\begin{itemize}
\item[(1)] $[X_i,X_j]=0$ is equivalent to ${[C_i,C_j]=0}$. This yields the condition $(i)$ of Definition~\ref{def:integrable-general}. 
\item[(2)] The matrix $D_{\xi_j}C_i$ being skew-symmetric means $X_i(F_j)=0$ which is the condition $(ii)$ of the same definition. 
\item[(3)] The condition $(iii)$ of Definition~\ref{def:integrable-general} holds due to the way we picked the first integrals and item $(1)$ above. 
\end{itemize}
So, $(X_1,..,X_p,F_1,..,F_q)$ is an integrable system in the sense of Definition~\ref{def:integrable-general}.  In each part, we will also show that $(X_1,..,X_p,F_1,..,F_q)$ is Hamiltonian integrable, according to Definition~\ref{def:integrable-hamiltonian}, w.r.t.  a big-isotropic structure. 
The matrix $C_1$ will be written as a linear combination of the shift matrix $H_m$, introduced in~\eqref{H-m}, and the identity matrix. Higher potentials of $H_m$ will be used to construct $C_2,\ldots,C_{p^\prime}$.  In the complex case we will use the complex version of $H_m$ which has, in its upper sub-diagonal, $I_2$ instead of $1$. 
 
Let rewrite the matrix $B$ as $B=\diag(\J(0),\J_{\ast})$ where
 \[\J_{\ast}=(\J(\pm \lambda_1), \ldots,\J(\pm \theta_1i),\ldots,\J(\pm(a_1\pm b_1i)),\ldots).\]
 We start with the generic case of 
 $X_\ast$.
 
\subsection*{The case of $X_\ast$:}
We divide this block itself into three types of sub-blocks i.e., $\J(\pm\lambda)$, $\J(\pm(a\pm bi)$ and $\J(\pm \theta i)$.  We begin with 
$$\J(\pm\lambda)=\scalemath{.9}{\diag(\J(\lambda),\J(-\lambda))=\diag(\lambda I_{r^+},J_{s^+_1}(\lambda),\ldots,J_{s^+_k}(\lambda),-\lambda I_{r^-},J_{s^-_1}(-\lambda),\ldots,J_{s^-_l}(-\lambda))},$$
where $\lambda$ is a nonzero real eigenvalue of $B$. 
Without any loss of generality, we assume that $r^++k\leq r^-+l$.  Considering non-degenerate eigenvalues as one dimensional Jordan blocks we pair each diagonal block of  $\J(\lambda)$ with a block of $\J(-\lambda)$. One gets the maximum number of integrals if this pairing is done in a way that minimises the sum of the differences between dimensions of the pairs. The extra blocks of eigenvalue $-\lambda$ are left alone. 
 The free elements of $D$ to generate our integrals are chosen in the following manner. For every pair $(J_\xi(\lambda),J_{\chi}(-\lambda))$
 \begin{itemize}
 \item[-] if $\dim(J_\xi(\lambda))=\dim(J_{\chi}(-\lambda))=1$ i.e. both blocks are non-degenerate eigenvalues, we pick the free element associated to them in $D_{r^+,r^-}$,
 \item[-] if $\dim(J_\xi(\lambda))=1$ and $\dim(J_{\chi}(-\lambda))=s^-_j\neq1$, we pick the associated free element to the non-degenerate eigenvalue $J_\xi(\lambda)$
 in the last column of $D_{r^+s_j^-}$, 
 \item[-] Similarly, $\dim(J_\xi(\lambda))=s^+_i\neq1$ and $\dim(J_{\chi}(-\lambda))=1$, we pick the associated free element in the last column of $D_{r^-s^+_i}$,
 \item[-] if  $\dim(J_\xi(\lambda))=s^+_i\neq 1$ and  $\dim(J_\chi(\lambda))=s^-_j\neq 1$, we pick all the free elements of $D_{s^+_i,s^-_j}$.
 \end{itemize}
  The form of the matrix $\Dd(0)$ shows that the chosen integrals are independent. Furthermore, the number of integrals associated to any pair $(J_\xi(\lambda),J_{\chi}(-\lambda))$ is 
  \begin{equation}\label{q-for-pairs}
  q_{\xi,\chi}=\min\{\dim(J_\xi(\lambda)),\dim(J_{\chi}(-\lambda))\}.
  \end{equation}
  For the extra blocks of the eigenvalue $-\lambda$ we do not associate any first integral. Therefore, Equality \eqref{q-for-pairs} holds for these left alone blocks as well. 
  
  Now, we introduce the commuting flows i.e., the matrices $C_1,...., C_p$.  We do so, for a pair $(J_\xi(\lambda),J_{\chi}(-\lambda))$ where $\dim(J_\xi(\lambda))=s^+_i\neq 1$ and  $\dim(J_\chi(\lambda))=s^-_j\neq 1$. The other cases, including the left alone blocks of the eigenvalue $-\lambda$,  are similar with the only difference that the shift matrix $H_m$, which will use in a moment, will be replaced by $\begin{pmatrix}\0& H_m\end{pmatrix}$ or $\begin{pmatrix}\0\\H_m\end{pmatrix}$. 
  
  Without loss of generality, we assume that $s^+_i\leq s^-_j$ i.e., $q_{\xi,\chi}=s^+_i$. The integral $(F_{s^+_i,s^-_j})_r=u^t(D_{s^+_i,s^-_j})_r u,\,\,r=1,\ldots,s^+_i$ are given by the free variables of $D_{s^+_i,s^-_j}$ defined in item {\bf (ii)}-$3$ of Theorem~\ref{diagonal-form-D}.  Let $c_0\ldots,c_{s^-_j-1}\in\mathbb{R}$ be arbitrary numbers. We define $C:=\diag (T_{s^+_i},T_{s^-_j})$  where 
  $$T_{s^+_i}=\sum_{e=0}^{s^+_i-1}c_{e}(H_{s^+_i})^e\quad \mbox{and}\quad T_{s^-_j}=\sum_{e=0}^{s^-_j-1}(-1)^{e+1} c_e(H_{s^-_j})^e.$$
 Now, let $E_l,\, l=0,..,s^-_j-1$  be the matrix obtained from $C$  setting $c_l=1$ and the rest of the coefficients  equal to zero. Clearly, we have $C_1=\lambda E_0+E_1.$ We set 
 \[ C_2:=E_1\,,....,\,C_{s^-_j}:=E_{s^-_j-1}.\]
 These commuting matrices  yield commuting independent  flows. Furthermore,  for every $l=1,\ldots,s^-_j$ and $r=1,..,s^+_i$, the matrix   $(D_{s^+_i,s^-_j})_rC_l$ is skew-symmetric. We have fulfilled all the conditions of List~\ref{list-1} therefore, the pair $(J_{s^+_i}(\lambda),J_{s^-_j}(-\lambda))$ is integrable of type $(s^-_j, s^+_i)$. As an example for $s^+_i=2$ and $s^-_j=3$ we have 
 \[D_{s^+_i,s^-_j}=\begin{pmatrix}0&0&0&0&d_1\\0&0&0&-d_1&d_2\\0&0&0&0&0\\0&-d_1&0&0&0\\d_1&d_2&0&0&0\end{pmatrix}\quad
  C=\begin{pmatrix}c_0&c_1&0&0&0\\0&c_0&0&0&0\\0&0&-c_0&c_1&-c_2\\0&0&0&-c_0&c_1\\0&0&0&0&-c_0\end{pmatrix}.\]
  
  The big-isotropic structure, in this case, is the graph of the presymplectic structure  $$\omega^\sharp=D^\prime(\J(\pm\lambda))^{-1}$$ where ${D^\prime=\sum_{r=1}^{q^\prime} (D_{s^+_i,s^-_j})_r}$. The calculation 
  \begin{align*}
  (D^\prime (\J(\pm\lambda))^{-1} C_l)^t&=C_l^t(D^\prime (\J(\pm\lambda))^{-1})^t=-C_l^t D^\prime (\J(\pm\lambda))^{-1}\\
  &=(-1)^2D^\prime C_l (\J(\pm\lambda))^{-1}=D^\prime (\J(\pm\lambda))^{-1}C_l,
  \end{align*}
where we used the facts that $D^\prime (\J(\pm\lambda))^{-1}$ and $(D_{s^+_i,s^-_j})_rC_l$ are skew-symmetric and ${[B,C_i]=0\Leftrightarrow [B^{-1},C_i]=0},$ shows that the matrix $D^\prime (\J(\pm\lambda))^{-1} C_i$ is symmetric. For every $l=1,\ldots,p^\prime$, we define the Hamiltonian $H_l$ by 
$$d_uH_i=D^\prime (\J(\pm\lambda))^{-1}C_lu=\omega^\sharp(X_l).$$
  Consequently, the system  $(X_1,..,X_{p^\prime},F_1,..,F_{q^\prime})$ is Hamiltonian integrable, according to Definition~\ref{def:integrable-hamiltonian}, w.r.t. the presymplectic structure $\omega$.

  The proof for a block of the form $\J(\pm(a\pm bi)$ is exactly the same as for the block $\J(\pm\lambda)$ with the only difference that we treat the $(2\times 2)$ blocks $B_{(a+bi)}$ like  numbers. Note that the matrices of  type $B_{(a+bi)}$ commute with each other.   The last case  is a block of the form 
 \[\J(\pm \theta i)=\diag(\oplus^rB_{(\pm \theta i)},J_{2s_1}(\pm \theta i),\ldots,J_{2s_k}(\pm \theta i)).\] 
For simplicity we assume that $k=1$, the case of $k\geq 2$ is similar. So, we set  $\J(\pm \theta i)=\diag(\oplus^rB_{(\pm \theta i)},J_{2s_1}(\pm \theta i).$
The matrix $\Dd(\pm\theta i)$ obtained by item {\bf (iii)} of Theorem~\ref{diagonal-form-D} has blocks of four types. For the first integrals in this case we do not use the types introduced in items  {\bf(iii)}-$2$ and {\bf (iii)}-$4$. From item {\bf (iii)}-$1$ we only choose the free variables $d_{11},\ldots,d_{rr}$ and all the free variables of item {\bf (iii)}-$3$. This way we obtain, in total, $q^\prime=r+s_1$ independent integrals. 
 Let $c_{11}\ldots,c_{rr}, \delta_{0},\ldots,\delta_{s_1-1}$, be arbitrary numbers. We define $$C:=\diag (C_{2r}, C_{2s_1})$$ where 
\begin{align*}
C_{2r}=\diag(C_{(\pm c_{11}i)},\ldots,C_{(\pm c_{rr}i)}),\quad C_{(\pm c i)}=\begin{pmatrix}0&c\\-c&0\end{pmatrix}
 \end{align*}
 and $C_{2s_1}=C_0+C_1$ where 
 \begin{align*}
&C_0=\sum_{l=0}^{k_0}\alpha_{2l+1}(H^c_{2s_j})^{(2l+1)},\\
&C_1=\sum_{l=0}^{k_1}D_{(\pm\delta_{2l}i)}(H^c_{2s_j})^{2l},
\end{align*}

 Similar to the case of $\J(\pm\lambda)$, let $E_z, z=1,\ldots,(r+k^\prime_1+1)$ be the matrices obtained from $C$ setting one coefficient  equal to one and the rest equal to zero.  The rest of the proof follows the same line of reasoning as in the case of $\J(\pm\lambda)$ and we get a Hamiltonian integrable system w.r.t. a big-isotropic structure which is the graph of a presymplectic form. To illustrate the situation, for the cases $s_1=4$ and $s_1=5$ the matrices $D_8$ and $D_{10}$ are the ones given in~\eqref{D8-D10-case(iii)}  and 
$$\scalemath{.7}{C_{8}=\begin{pmatrix}
C_{(\pm\delta_0i)}&\delta_1 I_2&C_{(\pm\delta_2 i)}&\delta_3 I_2\\
\0&C_{(\pm\delta_0 i)}&\delta_1 I_2&C_{(\pm\delta_2 i)}\\
\0&\0&C_{(\pm\delta_0 i)}&\delta_1 I_2\\
\0&\0&\0&C_{(\pm\delta_0 i)}
\end{pmatrix}},\quad\scalemath{.7}{C_{10}=\begin{pmatrix}
C_{(\pm\delta_0i)}&\delta_1 I_2&C_{(\pm\delta_2i)}&\delta_3 I_2&C_{(\pm\delta_4i)}\\
\0&C_{(\pm\delta_0i)}&\delta_1 I_2&C_{(\pm\delta_2i)}&\delta_3 I_2\\
\0&\0&C_{(\pm\delta_0i)}&\delta_1 I_2&C_{(\pm\delta_2i)}\\
\0&\0&\0&C_{(\pm\delta_0i)}&\delta_1 I_2\\
\0&\0&\0&\0&C_{(\pm\delta_0i)}
\end{pmatrix},\,.}$$

\subsection*{The case of $X_0=\J(0)u$:} We divide this block itself into four types of sub-blocks i.e., $\0_{r\times r}$, $J_{2k+1}(0)$,  pairs of even dimensional Jordan blocks $\diag(J_{2k}(0),J_{2k}(0))$ and single even dimensional blocks $J_{2k}(0)$. 
The block $\0_{r\times r}$ generates no dynamics.  We set the underlying big-isotropic structure to be the trivial structure $L_{(\0,\0)}$.  This way, we get $r$ independent Casimirs.  For the rest of the proof, note that any Jordan block $J_s(0)$ is actually the nilpotent shift matrix $H_s$. 

 For the odd dimensional block $J_{2k+1}(0)=H_{2k+1}$ the first integrals are defined by
  \begin{align*}
  &d_u F_1=(K_{2k+1}(1))u,\\
  &d_u F_2=(K_{2k+1}(1))H_{2k+1}^2u,\\
  &\vdots\\
 &d_u F_{k_i}=(K_{2k+1}(1))H_{2k+1}^{(2k_i-2)}u,\\
& d_u F_{(k_i+1)}=(K_{2k+1}(1))H_{2k+1}^{2k_i}u=u_{2k+1},
  \end{align*}
where  
$$\scalemath{.8}{K_{2k+1}(1)=\begin{pmatrix}0&0&\ldots&0&1\\0&0&\ldots&-1&0\\\vdots&\vdots&\iddots&\vdots&\vdots\\0&-1&\ldots&0&0\\1&0&\ldots&0&0\end{pmatrix}}.$$
 The commuting vector fields are
 \[X_1=H_{2k+1}u, X_2=H_{2k+1}^3u,...,X_{k_i}=H_{2k+1}^{(2k_i-1)}u.\]
It is very easy to check that $(X_1,...,X_{k_i},F_1,...,F_{k_i+1})$ is an integrable system in the sense of Definition~\ref{def:integrable-general}. Furthermore, 
 $(X_j,dF_j)\in L_{(H_{2k+1},K_{2k+1})}$ for ${j=1,...,k_i}$ and $(0,dF_{k+1})\in L_{(H_{2k+1},K_{2k+1})}$ i.e., this system is Hamiltonian integrable w.r.t. $L_{(H_{2k+1},K_{2k+1})}$. Since $K_{2k+1}$ is non-singular, the big-isotropic  structure $L_{(H_{2k+1},K_{2k+1})}$ is the graph of the Poisson structure 
 $\pi^\sharp= H_{2k+1}(K_{2k+1})^{-1}$.  In this setting, the function $F_{k_i+1}=\frac{1}{2}u^2_{2k+1}$ is a Casimir. 
 
 In the case of a pair of even dimensional Jordan blocks $$\diag(J_{2k}(0),J_{2k}(0))=\diag(H_{2k},H_{2k}),$$ the functions $$F_{2k}=u_{2k}, \quad F_{2k+1}=u_{4k},$$ are two Casimirs and  the rest of the first integrals are defined by
  \begin{align*}
  &d_u F_1=Wu,\\
  &d_u F_2=W\diag((H_{2k})^2,\0)u,\\
  &\vdots\\
  &d_u F_{k-1}=W\diag((H_{2k})^{(2k-2)},\0)u,\\
 &d_u F_{2k+1}=W\diag(\0,(H_{2k})^2)u,\\
 &\vdots\\
 &d_u F_{2k-1}=W\diag(\0,(H_{2k})^{(2k-2)})u,
 \end{align*}
 where
  $$W=\begin{pmatrix}\0&K_{2s}(1)\\(K_{2s}(1))^t&\0\end{pmatrix}.$$
  The commuting vector fields are
 \begin{align*}
 &X_1=\diag(H_{2k}, H_{2k})u,\, X_2=\diag((H_{2k})^3,H_{2k})u,\,\ldots, \, X_{k}=\diag((H_{2k})^{(2k-1)},H_{2k})u,\\
  &X_{k+1}=\diag(H_{2k},(H_{2k})^3)u,\, X_{(k+2)}=\diag(H_{2k},(H_{2k})^5)u,\ldots,\,X_{(2k-1)}=\diag(H_{2k},(H_{2k})^{(2k-1)})u.
  \end{align*}
  The underlying big-isotropic structure here is $L_{(\diag(H_{2k}, H_{2k}),W)}$ which is the graph of the Poisson structure $\pi^\sharp=(\diag(H_{2k}, H_{2k}))W^{-1}$.
  
Finally, for a single even dimensional Jordan block $J_{2k}(0)=H_{2k}$, the first integrals are defined by 
 \begin{align*}
 &d_u F_1=(K_{2k}(2))u,\,\,d_u F_2=(K_{2k}(2))(H_{2k})^2u,\ldots,\,F_{k}=(K_{2k}(2))(H_{2k})^{(2k-2)}u, \end{align*}
where 
\begin{equation}
 \scalemath{.8}{K_{2k}(2)=\begin{pmatrix}0&0&\ldots&0&0\\0&0&\ldots&0&1\\\vdots&\vdots&\iddots&-1&\vdots\\0&0&-1^{\iddots}&0&0\\0&1&\ldots&0&0\end{pmatrix}}.
 \end{equation}
  The commuting vector fields are
 \[X_1=H_{2k}u, X_2=(H_{2k})^3u,...,X_{k}=(H_{2k})^{(2s-1)}u.\]
  The underlying big-isotropic structure here is $L_{(H_{2k},K_{2k}(2))}$ which is not a Dirac structure. 
\end{proof}

As mentioned in the proof, for the general vector field $X=Bu$ where 
\begin{equation*}
 \scalemath{.9}{
B={\rm diag}(\J(0),\J(\lambda_1),\J(-\lambda_1), \ldots\ldots,\J(\pm \theta_1i),\J(a_1\pm b_1i),\J(-(a_1\pm b_1i)),\ldots),}
\end{equation*}
we put the first integrals, the commuting flows and the big-isotropic structures together in a diagonal form. For commuting flows we get vector fields of type $X=Cu$ where 
  \begin{equation}\label{matrix-C}
   \scalemath{.9}{
C=\diag(\Cc(0),\Cc(\lambda_1),\Cc(-\lambda_1),\ldots,\Cc(\theta_1 i),\Cc(-\theta_1 i),\ldots,\ldots,\Cc((a_1\pm b_1i)),\Cc(-(a_1\pm b_1i)),\ldots),}
\end{equation}
where the block diagonals are of a specific type $\sum_{l=0}^{s-1} c_l H^i_{s}$. In the literature, this type of matrices  are referred to as \emph{upper triangular Toeplitz matrices}. In fact, a matrix $C$ commutes with $B$ if and only if it is of the form~\eqref{matrix-C}, see \cite[Theorem $9.1.1$]{MR2228089}.

\section*{Acknowledgements}

We are grateful to the referees for their constructive input which substantially helped to improve the quality of the paper. 
 

\end{document}